\documentclass[12pt,a4paper]{amsart}

\usepackage[latin1]{inputenc}
\usepackage{amscd,latexsym,mathrsfs,amsfonts,amssymb,amsmath,amsthm,url}
\usepackage{color}
\usepackage{enumerate,verbatim}
\usepackage[pdftex,bookmarks,pdfnewwindow,plainpages=false,unicode,pdfencoding=auto]{hyperref}

\newtheorem{theorem}{Theorem}[section]
\newtheorem{conjecture}[theorem]{Conjecture}
\newtheorem{corollary}[theorem]{Corollary}
\newtheorem{lemma}[theorem]{Lemma}
\newtheorem{proposition}[theorem]{Proposition}
\newtheorem{definition}[theorem]{Definition}
\theoremstyle{remark}
\newtheorem*{remark}{Remark}

\numberwithin{equation}{section}

\newtheorem*{claim}{Claim}

\newcommand{\C}{\mathbb{C}}
\newcommand{\D}{\mathbb{D}}
\newcommand{\N}{\mathbb{N}}

\newcommand{\mcp}{\mathcal{P}}
\newcommand{\mcu}{\mathcal{U}}
\newcommand{\mcj}{\mathcal{J}}

\DeclareMathOperator*{\dist}{dist}

\providecommand{\keywords}[1]
{
  \small
  \textbf{\textit{}} #1
}

\begin{document}
\title[Shanks]{A counterexample to the Weak Shanks Conjecture}
\author[B\'en\'eteau]{Catherine B\'en\'eteau}
\address{Department of Mathematics, University of South Florida, 4202 E. Fowler Avenue,
Tampa, Florida 33620-5700, USA.} \email{cbenetea@usf.edu}
\author[Khavinson]{Dmitry Khavinson}
\address{Department of Mathematics, University of South Florida, 4202 E. Fowler Avenue,
Tampa, Florida 33620-5700, USA.} \email{dkhavins@usf.edu@usf.edu}
\author[Seco]{Daniel Seco}
\address{Departamento de An\'alisis Matem\'atico e IMAULL, Universidad de La Laguna, Avenida Astrof\'isico Francisco S\'anchez s/n, 38007 La Laguna (Santa Cruz de Tenerife), Spain.} \email{dsecofor@ull.edu.es}
\date{\today}

\begin{abstract}
We give an example of a function $f$ non-vanishing in the closed bidisk and the affine polynomial minimizing the norm of $1-pf$ in the Hardy space of the bidisk among all affine polynomials $p$. We show that this polynomial vanishes inside the bidisk. This provides a counterexample to the weakest form of a conjecture due to Shanks that has been open since 1980, with applications that arose from digital filter design. This counterexample has a simple form and follows naturally from \cite{BKLSS}, where the phenomenon of zeros seeping into the unit disk was already observed for similar minimization problems in one variable.

\end{abstract}

\maketitle

\section{Introduction}\label{Intro}

In the theory of invariant subspaces, a central role is played by the shift operator of multiplication by $z$ and its action on reproducing kernel Hilbert spaces. A celebrated case is that of the Hardy space over the unit disk, formed by functions whose Maclaurin coefficients are square-summable. There, a complete description of invariant subspaces is available, thanks to the work of Beurling (see, e.g., \cite{Ho}) that exploits the role played by so-called inner functions and leads to a characterization of cyclic vectors for the operator.  More than 80 years later, the equivalent problem in two dimensions is far from solved. When moving into the several complex variables world, a natural analogue of the one-dimensional shift is the tuple $(S_1, S_2)$ of shifts on each of the variables, given by \[S_1 (g)(z_1,z_2) = z_1 g(z_1, z_2); \quad S_2 (g)(z_1,z_2) = z_2 g(z_1, z_2).\] This is particularly relevant for the setting of the present article.  More precisely, let $\D = \{ z \in \C: |z|< 1 \}$ and let $\D^2 = \D \times \D.$  The Hardy space of the bidisk $H^2(\D^2)$ is the space of those holomorphic functions $g :  \D^2 \rightarrow \C$ given by a Taylor series around 0, \[g(z_1,z_2) = \sum_{k,l=0}^{\infty} a_{k,l} z_1^k z_2^l,\]
for which the norm \[\|g\|_{H^2(\D^2)} = \left( \sum_{k,l=0}^{\infty} |a_{k,l}|^2 \right)^{1/2}\]
is finite.  In this space, polynomials form a dense subclass, the shifts $S_1$ and $S_2$ are bounded operators (in fact, isometries) and there exists a \emph{reproducing kernel} at any point $(w_1,w_2) \in \D^2$. From the work of Rudin (see, e.g., \cite{Bunchofstuff}), it is known that the concept of inner function fails to completely describe the subspaces that are invariant under the shifts $S_1$ and $S_2$ simultaneously. We refer the reader to  \cite{Bunchofstuff} for more on this space and the theory of its invariant subspaces.

There is thus significant interest in the complex analysis community to understand cyclic functions (see below) in $H^2(\D^2)$.  With this in mind, the authors of the current article and their collaborators extended to this context the study of optimal approximation techniques that had been introduced in \cite{BCLSS1} for a family of one variable spaces. Recall that a function $g  \in H^2(\D^2)$ is called \emph{cyclic}, if $\mathcal{P}g$ is a dense subspace of $H^2(\D^2)$, where $\mathcal{P}$ is the space of all polynomials of two variables.  Equivalently, $g$ is cyclic if the smallest (closed) subspace of $H^2(\D^2)$ containing it and invariant under both $S_1$ and $S_2$ is the whole of $H^2(\D^2)$. Because polynomials are dense in $H^2(\D^2)$, it is then clear that the constant $1$ is a cyclic function, and therefore an alternative definition is that $g$ is cyclic if there exists a sequence of polynomials $\{p_n\}_{n \in \N}$ such that
\[\lim_{n \rightarrow \infty} \|1-p_n g\|_{H^2(\D^2)} =0.\]
In some earlier papers in the 1960s, such functions $g$ were also called weakly invertible (see, e.g., \cite{Sh,ASS} and references therein). 

Thus we set ourselves the task of understanding the problem of minimizing, for each $f \in H^2(\D^2)$ and each $n\in \N$,  the norm $\|1-pf\|_{H^2(\D^2)}$ over the set $\mathcal{P}_n$ of all polynomials of degree less than or equal to $n$. The norm minimizer is called an \emph{optimal polynomial approximant}, or \emph{OPA}, for $1/f$ 
 of degree $n$. Here one should choose a definition of degree or an enumeration of the monomials. In this paper, $\mathcal{P}_n$ will denote the set of 2-variable polynomials of $z_1$ and $z_2$ spanned by monomials $z_1^k z_2^l$ with $k+l \leq n$.  

Although OPAs as discussed above were introduced to study cyclic functions, they had been studied much earlier under a different name in connection with work on signal processing in engineering. In 1963, Robinson \cite{Ro} introduced the concept of ``least squares inverses": given a  sequence $a=(a_0, a_1, \ldots, a_m)$ of real numbers, find a sequence $b=(b_0, b_1, \ldots, b_n)$ of real numbers such that $b \star a$ (the discrete convolution of the sequences $b$ and $a$) approximates the ``unit spike", that is,
$$\| b \star a - (1, 0, \ldots, 0) \|_{\ell^2}$$ is of smallest norm.  Associating the sequence $a$ with the function $f(t) = \sum_{k=0}^m a_k e^{ikt}$ and $b$ with $p(t) = \sum_{k=0}^n b_k e^{ikt}$ shows that Robinson's problem is the same as finding the OPA of $1/f$ of degree $n$ in $H^2(\D)$ for a polynomial $f$. In 1980, Chui \cite{Ch} took this a step further and studied ``double least squares inverses": given a polynomial $f$, for each $n$, construct the $H^2(\D)$ OPA $p_n,$ then for each $k,$ construct the OPA $q_{n,k}$ of $p_n$. He proved in particular that the OPAs do not have zeros in the closed unit disk, and then a couple of years later, Chui and Chan \cite{CC} took advantage of this non-vanishing property to apply their ideas to recursive digital filter design.  In 1985, Izumino \cite{Iz} generalized the theory to arbitary $f \in H^2$ and reformulated the problem using operator theory. This circle of ideas was also linked to earlier work in the 1970s of Justice, Shanks and Treitel \cite{STJ} in the design of digital filters in function spaces of several variables. One goal of signal processing is to design filters that will transform a signal into different forms, say for compression or edge detection, and then, when needed, be able to reconstruct the original signal. In particular, if the signal lives in a space where a given function $f$ (the filter) is cyclic, then it is possible to completely reconstruct the signal without loss of information. Several authors were interested in whether this reconstruction would be \emph{stable}, which happened to be equivalent to the function $f$ generating a sequence of OPAs that do not vanish anywhere on the bidisk. In particular, Shanks and his co-authors \cite{STJ} conjectured in 1972 that for any polynomial $f$, the $H^2(\D^2)$ OPAs for $1/f$ are zero-free in $\D^2$, as is the case for one variable. Genin and Kamp \cite{GK} shortly thereafter gave a counterexample and constructed a polynomial $f$ with an $H^2(\D^2)$ OPA with a zero in $\D^2$. Their $f$ itself had a zero in $\D^2.$ However, in 1976, Anderson and Jury \cite{AJ} proved that the conjecture is true for certain polynomials of low and restricted degree. Two notable recent papers of M. Sargent and A. Sola in 2021 and 2022 \cite{SS1,SS2} simplified the Genin and Kamp  counterexample and discussed OPAs and orthogonal polynomials in a variety of spaces of several variables. See also \cite{BenCen} for a survey of OPAs of one and several variables and a more precise description of the filter design process. 

In 1980, Delsarte, Genin and Kamp \cite{DGK} stated the ``weakest" form of the Shanks Conjecture as follows:
\begin{conjecture}[The Weak Shanks Conjecture]
Suppose $f$ is a polynomial with no zeros in $\overline{\D^2}$. Then its OPAs are zero-free in $\D^2.$
\end{conjecture} 
The authors of \cite{RRS} thought that they have a counterexample to the Weak Shanks Conjecture, but Delsarte, Genin and Kamp showed that the proof failed in \cite{DGK2}, which was later also confirmed by Karivaratharajan and Reddy in \cite{RK}. In a recent preprint, Felder \cite{Fe} explores some specific cases where the conjecture holds. 

 In this article, we provide a counterexample to this conjecture. 

Our plan is as follows: In Section \ref{prelim}, we outline the argument, introduce a large family of function spaces of one variable, and explain our choice of a two-variable function $f$, taken by extending a one-variable function $F$ that played a special role in the study of the zeros of OPAs for one of those spaces. In Section \ref{prf1}, we show how to reduce the proof of existence of a counterexample (see Lemma \ref{lem1}) to checking an  explicit inequality (see Lemma \ref{lem2}). This inequality is numerically obvious yet we prove it analytically in Section \ref{prf2}. We conclude in Section \ref{further} with a few remarks on possible further directions of study and remaining unsolved questions.

\section{Preliminaries}\label{prelim}

\subsection{Outline of the argument}

From now on, we denote by $f$ the function $\D^2 \rightarrow \C$ given by \begin{equation}\label{eqn101}
f(z_1,z_2)= \left(1- \frac{z_1+z_2}{\sqrt{6}}\right)^{-\frac{5}{2}}.
\end{equation}
We also denote by $p_1$ the OPA of degree 1 to $1/f$, that is, the function of the form 
\begin{equation}\label{eqn102}
p(z_1,z_2)= \alpha + \beta z_1 + \gamma z_2,
\end{equation}
with $\alpha$, $\beta$ and $\gamma$ in $\C$ minimizing the norm 
\begin{equation}\label{eqn103}
\|1-pf\|_{H^2(\D^2)}
\end{equation}
among all such polynomials $p \in \mathcal{P}_1.$
The existence and uniqueness of $p_1$ is clear since $p_1f$ is, by definition, the orthogonal projection of $1$ onto the 3-dimensional subspace $\mathcal{P}_1 f$ of $H^2(\D^2)$. Because of the uniqueness of $p_1$, the fact that $f$ has real Taylor coefficients, the symmetry of $f$ with respect to the two variables, and the definition of the $H^2(\D^2)$ norm, it is clear that $p_1$ must be of the form  
\begin{equation}\label{eqn104}
p_1(z_1,z_2)= a + b (z_1 + z_2),
\end{equation}
for some choice of \emph{real} numbers $a, b$. From now on, we also fix this notation for the coefficients of $p_1$. Notice that $f$ satisfies the assumptions needed in the Weak Shanks Conjecture except for that of being a polynomial. However, since OPAs and their coefficients vary continuously within the space, and since the Taylor polynomials of $f$ converge uniformly to $f$ on compact subsets of $\sqrt{\frac{3}{2}} \cdot \D^2$, the domain of $f$, and in particular in the closed bidisk, in order to disprove the Weak Shanks Conjecture, it suffices to show the following claim regarding $p_1$.

\begin{claim}
 $p_1$  has a zero inside the bidisk.
\end{claim}
 This will imply that the Weak Shanks Conjecture is false for sufficiently large degree Taylor polynomials of $f$ around 0. From the form of $p_1$ obtained in \eqref{eqn104}, the restriction of $p_1$ to the pairs of real numbers on the diagonal is a continuous real function for $z_1=z_2 \in [-1,1]$ with values $p_1(1,1) = a+2b$ and $p_1(-1,-1) = a-2b$. In order to show the existence of a zero of $p_1$ inside the bidisk, we just need to show the following:

\begin{lemma}\label{lem1}
\[|a| < 2 |b|.\]
\end{lemma}

This yields:

\begin{corollary}
The Weak Shanks Conjecture is false.
\end{corollary}

\subsection{One variable spaces and zeros of OPAs}

Consider a sequence of positive weights $\omega:= \{\omega_k\}_{k=0}^{\infty}$ with $\omega_0 = 1,$
 $\lim_{k \rightarrow \infty} \frac{\omega_k}{\omega_{k+1}} = 1.$ 
The space $H_{\omega}^2$ consists of all analytic functions
$f\colon \D \rightarrow \C$ whose Taylor
coefficients in the expansion
\[f(z)=\sum_{k=0}^{\infty}a_kz^k, \quad z \in \D,\]
satisfy
\begin{equation*}
\|f\|^2_\omega=\sum_{k=0}^\infty |a_k|^2 \omega_k< \infty.
\end{equation*}
Given two functions $f(z) = \sum_{k=0}^{\infty}a_kz^k$ and $g(z) = \sum_{k=0}^{\infty}b_kz^k$ in $H^2_{\omega}(\D)$, we also have the associated inner product
\[ \langle f , g \rangle_{\omega} = \sum_{k=0}^\infty  a_k \overline{b_k} \omega_k . \]

\begin{definition}
Let $f \in H_{\omega}^2$. We say that a polynomial $p_n$ of degree at most $n$
is an \emph{optimal polynomial approximant (OPA)} of order $n$ to $1/f$ if $p_n$ minimizes
$\|p f-1\|_\omega$ among all polynomials $p$ of degree at most $n$.
\end{definition}
In other words, $p_{n}$ is an optimal polynomial approximant of order $n$ to $1/f$ if
\[  \|p_{n}f-1\|_{\omega}=\dist_{H_{\omega}^2}(1, f \cdot {\mathcal{P}}_n ), \]
where $ {\mathcal{P}}_n$ denotes the space of polynomials of degree at most $n.$ That is, $p_n f$ is the orthogonal projection of  $1$
onto the subspace $f\cdot {\mathcal{P}}_n$ and therefore, OPAs $p_n$ always exist and are unique for any nonzero function $f$, and any degree $n\geq 0$.
\begin{remark} Notice we already made use of the notation $\mathcal{P}_n$ for two-variable polynomials of algebraic degree at most $n$, but we will only look at such polynomials as restricted to the diagonal, where, if we identify $z=z_1=z_2$ they will coincide with 1-variable polynomials of the same degree. Thus we consider it an acceptable abuse of notation that will not lead to misunderstandings in the present study.\end{remark}

In \cite{BKLSS}, the authors studied the extremal problem 
\begin{align*}\label{iinf}
\inf_{n \in \N} \left\{|z|:p_n(z)=0,
\|p_nf-1\|_{\omega}=\min_{q\in\mcp_n}\|qf-1\|_{\omega},\,f\in
H^2_{\omega}\right\}.
\end{align*}
They noticed that in this one-variable setting, partly thanks to the Fundamental Theorem of Algebra, this extremal problem reduces to studying zeros of \emph{first order} (i.e., linear) polynomial approximants. Letting $p_1(z) = a+ bz$ be the first order approximant, noticing that $1-p_1f$ is orthogonal to $f$ and $zf$, solving the corresponding linear system for $a$ and $b$ and then setting $p_1(\zeta)= 0$ gives that  
the zero $\zeta$ of the first order optimal approximant
$p_1$ is given by \[\frac{\| z f \|^2_{\omega}}{ \langle f, zf \rangle_\omega }.
\]  
Thus, the extremal problem becomes to find
 \begin{equation}\label{extremal}
    \mcu:=\sup_{f \in H^2_{\omega}}  \frac{| \langle f, zf \rangle_\omega |}{\| z f \|^2_{\omega}}.
\end{equation} 
From now on, for a fixed weight $\omega$, we denote by $\mathcal{J}_\omega$ the Jacobi matrix with entries given by 
\begin{equation}\label{eqn201}
\left(\mathcal{J}_\omega\right)_{j,k} = 
\begin{cases}
\sqrt{\frac{\omega_j}{\omega_{j+1}}}, &\text{ if } k=j+1; \\

\sqrt{\frac{\omega_{j-1}}{\omega_{j}}},  &\text{ if } k=j-1;\\

0 &\text{ otherwise}.

\end{cases}
\end{equation}
The authors of \cite{BKLSS} were able to exploit the relationship between the extremal function, orthogonal polynomials on the real line, and norms of Jacobi matrices to prove the following theorem. 

\begin{theorem} Let $\omega= \{ \omega_k \}_{k=0}^{\infty}$ be such that $\omega_0 = 1,$ $\lim_{k \rightarrow \infty} \frac{\omega_k}{\omega_{k+1}} = 1.$ Then the following hold.
\begin{enumerate}
\item
\[ \sup_{f \in H_{\omega}^2}  \frac{| \langle f, zf \rangle_\omega |}{\| z f \|_{\omega}^2} = \frac{\|\mcj_{\omega}\|_{\ell^2 \rightarrow \ell^2}}{2}.\] 
\item If $\omega_k \leq \omega_{k+1},$ then $\|\mcj_{\omega}\|_{\ell^2 \rightarrow \ell^2} = 2$ and there is no solution to the extremal problem. 
\item If there exist $n,k \in \N$ such that $\omega_{k+n+1} < \frac{\omega_{k+1}}{4}$ (in particular, if the weights are strictly decreasing to $0$), then $\|\mcj_{\omega}\|_{\ell^2 \rightarrow \ell^2} > 2$, and the extremal problem has a solution.
\end{enumerate}
\end{theorem}
For certain particular cases of weights, they were able to find the extremal functions explicitly. This is the case when the weight is defining so-called \emph{Bergman-type spaces}, $A^2_\beta$ for $\beta > -1$. These are $H^2_\omega$ spaces with the choice of weights \[\omega_k := {k + 1 + \beta \choose k}^{-1}.\] For these weights, the (square of the) norm can also be written in its more standard integral expression
\[\|f\|^2_{A^2_\beta} = (\beta +1) \int_\D |f(z)|^2 (1-|z|^2)^\beta dA(z).\]

For a fixed weight $\omega$, extremal functions for \eqref{extremal} are unique up to a rotation of the variable and multiplication by a constant, so we restrict ourselves to  \emph{the} extremal function that has real coefficients and is normalized so that $f(0)=1$. A key finding in \cite{BKLSS} is Theorem 5.1, which includes the following description of the solution to \eqref{extremal}.
\begin{theorem}\label{A2beta}
The extremal function for \eqref{extremal} in $A^2_\beta$ is
\[f_\beta (z) := (1-z/c)^{-d},\]
where $c= \sqrt{\beta+2}$ and $d=\beta+3$.
\end{theorem}

\begin{remark}
Later on we will make use of a function $F$ that will be the function $f_{-1/2}$ in the notation of the above theorem. This will be useful because, in the next Section, we will show how $H^2(\D^2)$ contains a copy of $A^2_{-1/2}$ but equipped with an \emph{equivalent norm}. This suggests to look for the desired example through the lens of known examples in one variable. 
\end{remark}

\section{The core of the solution to the conjecture}\label{prf1}

In order to find a counterexample to the Weak Shanks Conjecture, let us look for a function of the form $f(z_1,z_2) = F \left( \frac{z_1+z_2}{2} \right),$ where $F \in H^2_{\omega}$ for some suitable weights. If $(z_1,z_2) \in \D^2$ then $\frac{z_1+z_2}{2} \in \D$. 
We would like to use this relationship to find a function $f(z_1,z_2)$ that is analytic in the closed bidisk and that has an OPA vanishing in $\D^2$. The needed weights are given by the following lemma. 

\begin{lemma}\label{weights}
Let $\omega_0=1$ and for $k \geq 1,$ let $\omega_k = \frac{{2k \choose k}}{2^{2k}}.$  Let $F \in H^2_{\omega}$.  If $f(z_1,z_2) = F \left( \frac{z_1+z_2}{2} \right),$ then $f \in H^2(\D^2)$ and 
\begin{equation*}
\langle f, f \rangle_{H^2(\D^2)}  = \langle F, F \rangle_{H^2_{\omega}}. 
\end{equation*}

\begin{proof}  Let $F(z) = \sum_{k=0}^{\infty} a_k z^k.$   
Note first that from the definition of the norm as an inner product, different powers of $(z_1+z_2)$ are mutually orthogonal in $H^2(\D^2)$ since they are linear combinations of monomials of different degrees. Thus we have right away that
\[\langle f, f \rangle_{H^2(\D^2)}  = \sum_{k=0}^\infty |a_k|^2 \left\langle  \left( \frac{z_1 +z_2}{2} \right)^k, \left( \frac{z_1 +z_2}{2} \right)^k \right\rangle_{H^2(\D^2)}.\]

Removing the constant terms on $2^{-k}$ from both sides, the inner product inside the sum of the right-hand side is equal to
\[2^{-2k} \left\langle  \left( \sum_{j=0}^k {k \choose j} z_1^jz_2^{k-j} \right), \left( \sum_{j=0}^k {k \choose j}  z_1^jz_2^{k-j} \right)\right\rangle_{H^2(\D^2)}.\]
For $j=0,...,k$, the monomials $z_1^j z_2^{k-j}$ are an orthonormal system, and hence
\[\langle f, f \rangle_{H^2(\D^2)}  = \sum_{k=0}^\infty 2^{-2k} |a_k|^2 \sum_{j=0}^k {k \choose j}^2.\]
The classical Chu-Vandermonde identity \cite{As} yields $\sum_{j=0}^k {k \choose j}^2 = {2k \choose k}$ and the definition of our weights, $\omega_k = \frac{{2k \choose k}}{2^{2k}}$ for $k \in \N$, completes the proof of the claimed identity.
From the above, it is clear that $F \in H^2_{\omega}$ if and only if $f \in H^2(\D^2),$ and the lemma is proved.
\end{proof}

\end{lemma}

Since the norms of these Hilbert spaces coincide, then the corresponding inner products also coincide, and therefore the two-variable extremal problem is the same as \eqref{extremal}.

By Stirling's Formula, one can show that the weights $\omega_k$ are comparable to 
$\frac{1}{\sqrt{k+1}},$ (in fact we will prove below a more precise statement), which implies that the space $H^2_{\omega}$  is equal to $A^2_{-1/2}$ as a set  for this collection of weights but has a different (yet equivalent) norm. By Theorem \ref{A2beta}, the function $F(z) = \left( 1-\frac{z}{\sqrt{\frac{3}{2}}}\right)^{-5/2}$ is extremal for the minimal zero problem in $A^2_{-1/2}$ and has a first degree OPA vanishing in $\D.$  However, since the norms are equivalent but not equal, we do not have a closed form for the actual extremal function in $H^2_{\omega}$, which means that although we know its OPAs will vanish inside the bidisk, we do not know that it is zero free itself. Thus although it is not immediate that this function will produce a counterexample to the Weak Shanks Conjecture, it seems like a good guess. 

 Therefore, let us consider the function $f$ defined in the Introduction,
$$f(z_1,z_2) = \left( 1 - \frac{\frac{z_1+z_2}{2}}{\sqrt{3/2}} \right)^{-5/2} = \left( 1 - \frac{z_1+z_2}{\sqrt{6}} \right)^{-5/2}.$$
We also use from now on the rest of the notation from the introduction regarding $p_1$, $a$, $b$ and $F$.

In order to check Lemma \ref{lem1} it then suffices to check that 
\begin{equation}\label{inequality}
| \langle F, zF \rangle_{\omega} | > \|zF\|^2_{\omega}.
\end{equation}

The proof of this inequality is a series of computations and estimates that involve: 
\begin{enumerate}
\item[(1)] A close examination of the weights $\omega_k$;
\item[(2)] Detailed estimates of the coefficients $a_k$ of the function $F$ that allow us to see how far we have to go in replacing $F$ by its Taylor polynomial to get an accurate enough estimate to ensure \eqref{inequality}; 
\item[(3)] Numerical verification of \eqref{inequality} for that Taylor polynomial. 
\end{enumerate}
Let us begin with the following lemma in order to address item (1) above. 

\begin{lemma}\label{weight_estimate}
For $k > 0,$ the product $\omega_k \cdot \sqrt{\pi k } \in (7/8,1)$, and 
$$\lim_{k \rightarrow \infty} \omega_k \cdot \sqrt{\pi k } =1.$$
\end{lemma}

\begin{proof}
A well-known estimate based on Stirling's Formula for the factorial \cite{HRob} gives that for $k \in \N \backslash \{ 0 \},$ 
\begin{equation*}
	\sqrt{2 \pi k} \left( \frac{k}{e} \right)^k e^{\frac{1}{12k+1}} < k! < \sqrt{2 \pi k} \left( \frac{k}{e} \right)^k e^{\frac{1}{12k}}.
\end{equation*}
Applying these estimates to $\omega_k = \frac{{2k \choose k}}{2^{2k}} = \frac{1}{2^{2k}} \cdot \frac{(2k)!}{(k!)^2},$ we deduce that for $k \geq 1,$
\begin{equation*}
\omega_k > 2^{-2k} \cdot \left(\sqrt{2 \pi 2k}\left (\frac{2k}{e}\right)^{2k} e^{\frac{1}{24k+1}} \right) \cdot \left(\frac{(\frac{e}{k})^{2k} e^{\frac{-1}{6k}}}{2\pi k} \right) = \frac{e^{\frac{1}{24k+1}-\frac{1}{6k}}}{\sqrt{\pi k}}.\end{equation*}
Notice that $1/(24k+1) - 1/6k$ is a increasing function of $k \geq 1$ and that $e^{1/25-1/6} > 0.88 > 7/8$ which gives
\begin{equation*} \omega_k > \frac{7}{8 \sqrt{\pi k}}.
\end{equation*}

From the other side, we obtain
\begin{equation*}
\omega_k < \frac{e^{\frac{1}{24k}-\frac{1}{6k+1/2}}}{\sqrt{\pi k}} < \frac{1}{\sqrt{\pi k}}. 
%\frac{e^{-\frac{1}{10k}}}{\sqrt{\pi k}} < \frac{1}{\sqrt{\pi k}}.
\end{equation*}
Using the more precise inequalities above for large $k$ gives that $$\lim_{k \rightarrow \infty} \omega_k \cdot \sqrt{\pi k } =1$$ and using the outermost inequalities above give that $\omega_k \cdot \sqrt{\pi k } \in (7/8,1)$, as desired. 
\end{proof}

Next, we find the detailed estimates of the coefficients to allow for a specific truncation of the power series of $F$ to suffice. It turns out that computing the inner products in \eqref{inequality} using only coefficients with indices $j=0,...,25$ of $F$ will yield enough accuracy. For the rest of the article, we let
\[F(z)  = \sum_{j=0}^{\infty} a_jz^j = \left( 1- \frac{z}{\sqrt{3/2}} \right)^{-5/2}.\]
The quantity \eqref{extremal} that we want to evaluate for $F$ will be computed from the 3 quantities $S_1$, $S_2$ and $S_3$, where
\begin{align*}
S_1 &:= \sqrt{\frac{2}{3}}\left( \sum_{j=0}^{24} \left(1 + \frac{3}{2(j+1)}\right) |a_j|^2 \omega_{j+1} \right),\\ 
S_2 &:= \sum_{j=0}^{24} |a_j|^2 \omega_{j+1},\\
S_3 &:= (5-\sqrt{6}) |a_{25}|^2 \omega_{26}.\end{align*}
Now we are ready to formulate a key step of the proof:
\begin{proposition}
The function $F$  satisfies
\begin{equation}\label{recurrence}
a_j = \prod_{t=1}^{j} \left[\sqrt{\frac{2}{3}} \left( 1 + \frac{3}{2t} \right) \right].
\end{equation} 
Moreover, 
\begin{equation}\label{eqn501}| \langle F, zF \rangle_{\omega}| > S_1+  \sqrt{6} |a_{25}|^2 \omega_{26}\end{equation}

while 

\begin{equation}\label{eqn502}\|zF\|_{\omega}^2 < S_2 + 5 |a_{25}|^2 \omega_{26}.\end{equation}
\end{proposition}

\begin{proof}
Notice that, from the definition of binomial coefficients,
\begin{equation}\label{aj}
a_j = \left( -\sqrt{2/3} \right)^j {-5/2 \choose j } = \prod_{t=1}^{j} \left[
\sqrt{\frac{2}{3}} \left( 1 + \frac{3}{2t} \right) \right].
\end{equation} 

From \eqref{aj}, it is clear that for $j \geq 24,$ 
\begin{equation}\label{aj_ineq}
\frac{|a_{j+1}|^2}{|a_j|^2} = (2/3) \cdot \left(1+ \frac{3}{2(j+1)}\right)^2 \leq (2/3) \cdot \left(\frac{53}{50}\right)^2 < \frac{3}{4},
\end{equation}
where the last estimate holds because \[(53/50)^2=(1.06)^2=1.1236<1.125 =9/8.\]
We can also deduce from \eqref{aj} that for all $j \in \N$ 
\[|a_{j+1}|^2 \geq \frac{2}{3} |a_j|^2.\]
However, we will be a bit more ambitious and use that for $j \geq 25$ we have
\[|a_{j+1}|^2 \omega_{j+1} \geq |a_j|^2 \omega_j C,\]
where $C$ is given by 
\[C= \frac{2}{3} \inf_{j \geq 25} \frac{\omega_{j+1}}{\omega_j}.\]
To obtain $C$ notice that
\begin{equation}\label{eqn301}
\frac{\omega_{j+1}}{\omega_j} = \frac{2^{2j}{2j+2  \choose j+1}}{2^{2(j+1)}{ 2j \choose j}} = \frac{j+1/2}{j+1},\end{equation}
which is increasing. 
Therefore
\begin{equation}\label{Cest}
C= \frac{2}{3} \cdot \frac{51}{52} = \frac{17}{26}.
\end{equation} 

In order to show the estimate from below for $|\langle F, zF \rangle_\omega |$ we just need to show that the tail of this inner product is controlled accordingly, depending on the term $|a_{25}|^2$. Indeed, from the definition of the inner product and applying \eqref{aj} to $a_{j+1}\bar{a_j}$, we obtain:
\[| \langle F, zF \rangle_{\omega}|  \geq \sqrt{\frac{2}{3}}\left( \sum_{j=0}^{\infty} \left(1 + \frac{3}{2(j+1)}\right) |a_j|^2 \omega_{j+1} \right).\]
But all the terms in the sum with $j \geq 25$ can be bounded from below by $\frac{53}{50}\omega_{26} |a_{25}|^2 C^{t}$ for $t=j-25$. Since $C<1$, that means that the tail contributes at most
\[\frac{53}{50} |a_{25}|^2 \omega_{26} \sum_{t=0}^{\infty} C^{t} = \frac{689}{225} |a_{25}|^2\omega_{26} > 3 |a_{25}|^2\omega_{26}.\]
This completes the lower estimate. It remains to show that
\[\sum_{j=26}^{\infty}|a_j|^2\omega_{j+1} < 4 |a_{25}|^2\omega_{26},\] but of course for $j \geq 26$ we have $\omega_{j+1}< \omega_{26}$ and the exponential decay of $|a_j|^2$  in \eqref{aj_ineq} is sufficient for this.

\end{proof}

From Proposition \ref{recurrence}, a sufficient condition for \eqref{inequality} is that the right-hand side of \eqref{eqn501}, $ S_1+  \sqrt{6} |a_{25}|^2 \omega_{26}$, is greater than that of \eqref{eqn502}, $ S_2 + 5 |a_{25}|^2 \omega_{26}$. This means that we only need to check the following inequality:

\begin{lemma}\label{lem2}
\[S_1>S_2 + S_3.\]
\end{lemma}

All 3 terms can be computed exactly as algebraic numbers in finite time analytically, but perhaps for some readers it will be enough at this point to mention that numerically $S_1 \approx 42.07$ while $S_2 \approx 41.04$ and $S_3 \approx 0.11$, and so the uncertainty of the result is much smaller than the margin we have. However, we understand this would not constitute a complete proof without the next section, where we will check analytically that Lemma \ref{lem2} is valid.

\section{Checking the finite sum condition}\label{prf2}

Now we proceed to prove Lemma \ref{lem2}.  For $j=0,...,25$, denote by \[H_j = |a_j|^2 \omega_{j+1}, \quad q_j = \frac{(2j+1)(2j+3)^2}{12j(j+1)^2}. \] It will also be useful to denote 
\[S_4 =\sum_{j=0}^{24} \frac{H_j}{j+1}.\] First notice that \begin{equation}\label{eqn401}\sqrt{\frac{3}{2}} S_1 = S_2 + \frac{3}{2} S_4,\end{equation} and that from \eqref{aj_ineq} and \eqref{eqn301} we have 
\[H_j = q_j H_{j-1}, \quad j \geq 1.\] 
Let us determine $H_j$ for each $j$ with sufficient accuracy. Notice that, since $q_j$ is rational and $H_0=1/2$, all $H_j$ are rational. The denominator in $q_j$ can only simplify with a factor from the numerator if $j$ or $3$ have a common factor with $(2j+1)$ or $(2j+3)$. That can only happen when $j$ is congruent to $0$ or $1 \mod 3$ and then only multiples of 3 can cancel out. As a result the simplified rational and approximate list of values of $q_j$ together with the inferred lower and upper bounds for $H_j$ and $H_j / (j+1)$ are indicated in the table below (where all $q_j$ approximations are truncations to 3 digits after the decimal, so they are exact within 0.001 of the correct value and bound that correct value from below).
\begin{table}[h!]
\centering
\begin{tabular}{c | c c c c c}
\hline
j & Numer. $q_j$ & Denom. $q_j$ & $q_j \approx$ & $H_j \in$  & $H_j / (j+1) \in$  \\
\hline
0 & . & . & . &  $[0.500,0.500]$ & $[0.500,0.500]$ 
\\ 1 & 25 & 8 & 3.125 &  [1.562,1.563] & $[0.781,0.782]$
\\ 2 & 245 & 144  & 1.701 & [2.656,2.661] & $[0.885,0.887]$ 
  \\ 3 &  21 & 16  &  1.312 &  [3.484,3.494] & $[0.871,0.874]$ 
\\  4 &   363 & 320 &  1.134 & [3.950,3.966] & $[0.790,0.794]$ 
  \\ 5 & 1859 & 1800  & 1.032   & [4.076,4.097] & $[0.679,0.683]$ 
 \\ 6 &  325 & 336  &  0.967 & [3.941,3.966]  & $[0.563,0.567]$
 \\ 7 &  1445  & 1568   &  0.921 & [3.629,3.657]  & $[0.453,0.458]$
  \\  8 &  6137 & 6912 &  0.887  & [3.218,3.248]  & $[0.357,0.361]$
 \\ 9 &  931 & 1080  &  0.862  & [2.773,2.804]  & $[0.277,0.281]$
 \\ 10 &  3703 & 4400  &  0.841  & [2.332,2.361]  & $[0.212,0.215]$
 \\ 11 &  14375 & 17424  &  0.825 &  [1.923,1.951]  & $[0.160,0.163]$
\\ 12 &  675 & 832  &  0.811  & [1.559,1.585]  & $[0.119,0.122]$
\\ 13 &   7569 & 9464  &  0.799 & [1.245,1.268]  & $[0.088,0.091]$
  \\ 14 &  27869 & 35280   &  0.789 & [0.982,1.002]  & $[0.065,0.067]$
 \\ 15 &  3751 & 4800  &  0.781  & [0.766,0.784]  & $[0.047,0.049]$
 \\ 16 &  13475 & 17408   &  0.774  & [0.592,0.608]  & $[0.034,0.036]$
\\ 17 &  47915 & 62424  &  0.767  & [0.454,0.467]  & $[0.025,0.026]$
\\ 18 &   6253 & 8208  & 0.761 & [0.345,0.356]  & $[0.018,0.019]$
\\ 19 &  21853 & 28880  &  0.756  & [0.260,0.270]  & $[0.013,0.014]$
\\ 20 &  75809 & 100800  & 0.752 & [0.195,0.204]  & $[0.009,0.010]$
\\ 21 &  3225 & 4312  & 0.747 & [0.145,0.153]  & $[0.006,0.007]$
\\ 22 &   33135  & 44528 & 0.744  & [0.107,0.114]  & $[0.004,0.005]$
 \\  23 & 112847 & 152352 & 0.740  & [0.079,0.085]  & $[0.003,0.004]$
 \\ 24 &  14161 & 19200  & 0.737  & [0.058,0.063]  & $[0.002,0.003]$
\\ 25 & 47753 & 65000  & 0.734 & [0.042,0.047]  & $[0.001,0.002]$
 \end{tabular} \end{table}

Summing the lower and upper bounds from $j=0$ to $24$ we can infer from these bounds that $S_2 \in [40.831, 41.227]$, that $S_3 \in [0.107,0.120]$ and that \[\sum_{j=0}^{24} \frac{H_j}{j+1} \in [6.961,7.018].\]
We make use of \eqref{eqn401}, and derive that $S_1 > S_2 + S_3$ happens if and only if
\begin{equation}\label{eqn402}\sqrt{\frac{3}{2}} S_4 > \left(1- \sqrt{\frac{2}{3}}\right) S_2 + S_3.\end{equation}

From the estimates we obtain and using $\sqrt{1.5}\in [1.224,1.225]$ and $\sqrt{2/3} \in [0.816,0.817]$, we see that the left-hand side of \eqref{eqn402} is bounded from below by $8.525$, while its right-hand side is bounded above by $7.586 + 0.120 = 7.706$. This implies that $S_1 - S_2 - S_3 > 0.819$ and concludes the proof.

\newpage

\section{Further work}\label{further}

\bigskip

Here we want to emphasize that our refutal of the Weak Shanks Conjecture only opens the door to a very rich collection of important problems waiting to be investigated. 
For instance, even though we have proved that the zero sets of the OPAs for our function $f$ do get inside the bidisk, we don't know how close to the origin the complex line of zeros for $p_1$ gets nor whether the OPAs of higher degree for that function  have zeros in the bidisk. More generally, we don't know \emph{how far} inside the bidisk zeros of any OPA can get. It would be interesting to study that extremal problem, as we did for one variable in \cite{BKLSS}, as we vary through all orderings of polynomials and/or through all functions in $H^2(\D^2)$. However, we cannot stress enough the great difference between the one and two variable problems there is as a result of the lack of a Fundamental Theorem of Algebra: the fact that zeros are not isolated means that the problem does not reduce to studying any sort of low dimensional polynomials. Moreover, one needs to choose what ``how far'' means in this context: is this the point on the zero set of minimal $\ell_1$, $\ell_2$, or $\ell_\infty$ norm in $\C^2$?

On the other hand, the embedding of functions depending only on $z_1+z_2$ proved useful in our study, partly because there is a natural version of the shift that creates a set of mutually orthogonal classes of monomials.  Perhaps other closed subspaces can take the zeros further inside the disk, but it is not evident to us which spaces one can use. 

One could also consider finding counterexamples with other orderings of monomials.  
If we choose an ordering in which $1$, $z_1$ and $z_2$ are not the first 3 elements of the basis, or an ordering which does yield a symmetric OPA, we may get a different OPA, even for our function $f$. Is there any particular choice of ordering of the monomials for which the sequence of OPAs must be zero-free? It seems unlikely but far from our current reach.  

What happens in higher dimensional polydisks? We do know that our result implies a similar failure of OPAs to be zero-free in higher dimensional polydisk Hardy spaces, from embedding remarks in \cite{PJM} but do the answers to any of the other questions we pose vary in higher dimensions?

Perhaps most interesting, the phenomenon observed here (OPAs having zeros inside the domain for zero-free functions) can also be found in subspaces of $H^2(\D^2)$ called Dirichlet-type spaces $D_\alpha(\D^2)$.  These spaces are defined in a similar fashion to $H^2(\D^2)$ but the norm of $z_1^j z_2^k$ is given by $(j+1)^{\alpha/2} (k+1)^{\alpha/2}.$  For a fixed $f$, the inner products $\langle z_1^j z_2^k f, z_1^l z_2^m f \rangle_{\alpha}$ all vary continuously in terms of the parameter $\alpha$.  Hence for $\alpha >0$ small enough, OPAs for $f$ in $D_\alpha(\D^2)$ will have zeros inside $\D^2.$ How far into the positive $\alpha$ values can we go? Is there a finite supremum for the values of $\alpha$ for which this phenomenon occurs?

What happens in the Hardy space over the unit ball of $\C^n$? 
Is the extremal for the symmetric functions space we obtained actually zero-free? 

Finally, in many Dirichlet-type spaces in $\C$, the zeros of OPAs eventually escape every compact subset of the unit disk, due to the connection between OPAs and reproducing kernels in weighted spaces (see \cite{BKLSS}). Does this behavior prevail in the bidisk? Along the same lines, does Jentzsch's phenomenon related to accumulation of zeros of OPAs proven in \cite{BKLSS} extend to higher dimensions? We hope to return to some of these questions in the future.

\bigskip

\noindent\textbf{Acknowledgements.}   This work has been funded through the grant PID2023-149061NA-I00 by the Generaci\'on de Conocimiento programme and through grant RYC2021-034744-I by the Ram\'on y Cajal programme from Agencia Estatal de Investigaci\'on (Spanish Ministry of Science, Innovation and Universities),  and Simon's Foundation grant 513381.

\end{document}